\documentclass[11pt]{amsart}

\usepackage[mathscr]{eucal}
\usepackage{amsmath,amssymb,amsfonts,amsthm,enumerate}
\usepackage{hyperref}

\newtheorem{theorem}{Theorem}[section]
\newtheorem{lemma}[theorem]{Lemma}
\newtheorem{corollary}[theorem]{Corollary}
\newtheorem{proposition}[theorem]{Proposition}
\newtheorem{conjecture}[theorem]{Conjecture}
\newtheorem{theirtheorem}{Theorem}

\newcommand{\Z}{\mathbb Z}

\DeclareMathOperator{\ord}{ord}

\DeclareMathOperator{\supp}{Supp}

\newcommand{\la}{\langle}
\newcommand{\ra}{\rangle}
\newcommand{\be}{\begin{equation}}
\newcommand{\ee}{\end{equation}}
\newcommand{\und}{\;\mbox{ and }\;}

\newcommand{\ber}{\begin{eqnarray}}
\newcommand{\eer}{\end{eqnarray}}
\newcommand{\Sum}[2]{\underset{#1}{\overset{#2}{\sum}}}
\newcommand{\Summ}[1]{\underset{#1}{\sum}}

\newcommand{\wtilde}{\widetilde}

\newcommand{\Fc}{\mathcal F}
\newcommand{\vp}{\mathsf v}
\newcommand{\h}{\mathsf h}

\DeclareSymbolFont{goo}{OMS}{cmsy}{b}{n}
\DeclareMathSymbol{\gooT}{\mathalpha}{goo}{"1}
\newcommand{\bdot}{\mathbin{\gooT}}
\begin{document}

\title{A  Multiplicative Property for Zero-Sums II}
\author{David J. Grynkiewicz}
\address{Department of Mathematical Sciences\\ University of Memphis\\ Memphis, TN 38152\\
USA}
\email{diambri@hotmail.com}
\author{Chao Liu}

\address{Department of Mathematical Sciences\\ University of Memphis\\ Memphis, TN 38152\\
USA}
\email{chaoliuac@gmail.com}


\begin{abstract}For $n\geq 1$, let $C_n$ denote a cyclic group of order $n$.
Let $G=C_n\oplus C_{mn}$ with $n\geq 2$ and $m\geq 1$, and let $k\in [0,n-1]$. It is known that any sequence of $mn+n-1+k$ terms from $G$ must contain a nontrivial zero-sum of length at most $mn+n-1-k$. The associated inverse question is to characterize those  sequences with maximal length $mn+n-2+k$ that fail to contain a nontrivial zero-sum subsequence of length at most $mn+n-1-k$. For $k\leq 1$, this is the inverse question for the Davenport Constant. For $k=n-1$, this is the inverse question for the $\eta(G)$ invariant concerning short zero-sum subsequences. The structure in both these cases is known, and the structure for $k\in [2,n-2]$ when $m=1$ was studied previously with it  conjectured that they must have the form $S=e_1^{[n-1]}\boldsymbol{\cdot} e_2^{[n-1]}\boldsymbol{\cdot} (e_1+e_2)^{[k]}$ for some basis $(e_1,e_2)$, with the conjecture established in many cases.  In this paper,
we focus on the case $m\geq 2$. Assuming the conjectured structure holds for $k\in [2,n-2]$ in $C_n\oplus C_n$, we characterize the structure of all sequences of maximal length $mn+n-2+k$ in $C_n\oplus C_{mn}$ that fail to contain a nontrivial zero-sum of length at most $mn+n-1-k$, showing they must have either have the form $S=e_1^{[n-1]}\boldsymbol{\cdot} e_2^{[sn-1]}\boldsymbol{\cdot} (e_1+e_2)^{[(m-s)n+k]}$ for some $s\in [1,m]$ and basis $(e_1,e_2)$ with $\mathsf{ord}(e_2)=mn$, or else have the form $S=g_1^{[n-1]}\boldsymbol{\cdot} g_2^{[n-1]}\boldsymbol{\cdot} (g_1+g_2)^{[(m-1)n+k]}$ for some generating set $\{g_1,g_2\}$ with $\mathsf{ord}(g_1+g_2)=mn$.
In view of prior work, this reduces the structural characterization for a general rank two abelian group to the case $C_p\oplus C_p$ with $p$ prime.
Additionally, we give a new proof of the precise structure in the case $k=n-1$ for $m=1$. Combined with known results, our results unconditionally establish the  structure of extremal sequences  in $G=C_n\oplus C_{mn}$ in many cases, including when  $n$ is only divisible by primes at most $7$, when    $n\geq 2$ is a prime power and $k\leq \frac{2n+1}{3}$, or when  $n$ is composite and $k=n-d-1$ or $n-2d+1$ for a proper, nontrivial divisor $d\mid n$.
\end{abstract}

\maketitle

\section{Introduction and Preliminaries}

Regarding combinatorial notation for sequences and subsums, we utilize the standardized system surrounding multiplicative strings as outlined in the references \cite{Ger-Ruzsa-book} \cite{Alfredbook} \cite{Gbook}. For the reader new to this notational system, we begin with a self-contained review.

\subsection*{Notation} All intervals will be discrete, so for $x,\,y\in \Z$, we have $[x,y]=\{z\in \Z:\; x\leq z\leq y\}$.
For integers $x$ and $n$ with $n\geq 1$, let $(x)_n\in [0,n-1]$ denote the least non-negative representative for $x$ modulo $n$.
 We use $C_n$ to denote a cyclic group of order $n$. A finite abelian group $G$ has the form $G=C_{n_1}\oplus\ldots\oplus C_{n_r}$ with $1<n_1\mid \ldots\mid n_r$, where $n_r=\exp(G)$ is the \textbf{exponent} and $r\geq 0$ is the \textbf{rank} of $G$, which is the minimal cardinality of a generating set for $G$. For $r\leq 2$, an arbitrary rank at most two abelian group has the form $G=C_n\oplus C_{mn}$ with $n,m\geq 1$. When $n\geq 2$, so the rank $r=2$,
a (ordered) \textbf{basis} for $G$ is a pair $(e_1,e_2)$ of elements $e_1,e_2\in G$ such that $G=\la e_1\ra\oplus \la e_2\ra=C_n\oplus C_{mn}$.

Let $G$ be an abelian group. In the tradition of Combinatorial Number Theory, a sequence of terms from $G$ is a finite, unordered string of elements from $G$. We let $\Fc(G)$ denote the free abelian monoid with basis $G$, which consists of all (finite and unordered) sequences $S$ of terms from $G$ written as multiplicative strings using the boldsymbol $\bdot$ . This means a sequence $S\in \Fc(G)$ has the form  $$S=g_1\bdot\ldots\bdot g_\ell$$ with $g_1,\ldots,g_\ell\in G$ the terms in $S$.
Then $$\vp_g(S)=|\{i\in [1,\ell]:\; g_i=g\}|$$ denotes the multiplicity of the terms $g$ in $S$, allowing us to represent a sequence $S$ as $$S={\prod}^\bullet_{g\in G}g^{[\vp_g(S)]},$$ where $g^{[n]}={\underbrace{g\bdot\ldots\bdot g}}_n$ denotes a sequence consisting of the term $g\in G$ repeated $n\geq 0$ times.
 The maximum multiplicity of a term of $S$ is the height of the sequence, denoted $$\h(S)=\max\{\vp_g(S):\; g\in G\}.$$ The support of the sequence $S$ is the subset of all elements of $G$ that are contained in $S$, that is, that occur with positive multiplicity in $S$, which is denoted $$\supp(S)=\{g\in G:\; \vp_g(S)>0\}.$$
 The length of the sequence $S$ is $$|S|=\ell=\Summ{g\in G}\vp_g(S).$$  A sequence $T\in \Fc(G)$ with $\vp_g(T)\leq \vp_g(S)$ for all $g\in G$ is called a subsequence of $S$, denoted $T\mid S$, and in such case, $S\bdot T^{[-1]}=T^{[-1]}\bdot S$ denotes the subsequence of $S$ obtained by removing the terms of $T$ from $S$, so $\vp_g(S\bdot T^{[-1]})=\vp_g(S)-\vp_g(T)$ for all $g\in G$.

 Since the terms of $S$ lie in an abelian group, we have the following notation regarding subsums of terms from $S$. We let $$\sigma(S)=g_1+\ldots+g_\ell=\Summ{g\in G}\vp_g(S)g$$ denote the sum of the terms of $S$ and call $S$ a \textbf{zero-sum} sequence when $\sigma(S)=0$. A \textbf{minimal zero-sum} sequence is a zero-sum sequence that cannot have its terms partitioned into two proper, nontrivial zero-sum subsequences.  For $n\geq 0$, let
 \begin{align*}&\Sigma_n(S)=\{\sigma(T):\; T\mid S, \; |T|=n\},\quad
 \Sigma_{\leq n}(S)=\{\sigma(T):\; T\mid S, \; 1\leq |T|\leq n\},\quad\und\quad\\
 &\Sigma(S)=\{\sigma(T):\; T\mid S, \; |T|\geq 1\}
\end{align*}
 denote the variously restricted collections of  subsums of $S$.
 The sequence $S$ is \textbf{zero-sum free} if $0\notin \Sigma(S)$.
 Finally, if $\varphi:G\rightarrow G'$ is a map, then $$\varphi(S)=\varphi(g_1)\bdot\ldots\bdot \varphi(g_\ell)\in \Fc(G')$$ denotes the sequence of terms from $G'$ obtained by applying $\varphi$ to each term from $S$.

\subsection*{Background} Let $G$ be a finite abelian group.
The Davenport Constant for $G$ is the minimal  integer $\mathsf D(G)$ such that any sequence of terms from $G$ with length at least $\mathsf D(G)$ must contain a nontrivial zero-sum subsequence. Equivalently, $\mathsf D(G)$ is the maximal length of a minimal zero-sum sequence (see \cite{chaoI} \cite{Alfredbook}). Besides being of interest as an independent topic in Combinatorial Number Theory, it also plays an important role when studying factorization  in Krull Domains and, more generally, in (Transfer) Krull Monoids. See \cite{Alfredbook} \cite{Ger-Ruzsa-book}.
For a general rank at most two abelian group $G=C_n\oplus C_{mn}$, where  $m,\,n\geq 1$, we have \cite[Theorem 5.8.3]{Alfredbook} $$\mathsf D(C_n\oplus C_{mn})=mn+n-1.$$ This is a classical result of Olson  \cite{olson-rk2} or van Emde Boas and Kruyswijk \cite{Boas-K-rk2-Dav} whose proof requires the constant $\eta(G)$, defined as the  minimal length  such that any sequence of terms from $G$ with length at least $\eta(G)$ contains a nontrivial zero-sum subsequence of length at most $\exp(G)$.
For rank at most two groups, we have \cite{olson-rk2} \cite{Boas-K-rk2-Dav} \cite[Theorem 5.8.3]{Alfredbook} $$\eta(C_n\oplus C_{mn})=mn+2n-2.$$
Specializing a particular case of a more general invariant \cite{Cohen} \cite{Cong-zs}, Delorme, Ordaz and Quiroz introduced \cite{S_<k-invariant-origin-delorme-ordaz} the  constant $\mathsf s_{\leq \ell}(G)$ defined as the minimal length such that $$|S|\geq \mathsf s_{\leq \ell}(G)\quad \mbox{ implies }\quad 0\in \Sigma_{\leq \ell}(S).$$
For connections with Coding Theory, see \cite{Cohen}. The constant $\mathsf s_{\leq \ell}(G)$ has also been studied in various other contexts  \cite{freeze}
\cite{Roy-s<k} \cite{Gao-liu}.

Since  $\mathsf s_{\leq \ell}(G)=\infty$ for $\ell<\exp(G)$ and coincides with the invariants $\eta(G)$ and  $\mathsf D(G)$ for the values $\ell=\exp(G)$ and $\ell=\mathsf D(G)$, it may be viewed as a means of  interpolating these constants as $\ell\in [\exp(G),\mathsf D(G)]$.
For the case of rank two groups, Chulin Wang and Kevin Zhao determined its exact value \cite{kevin-S_<k-invariant}: \begin{align*}\mathsf s_{\leq mn+n-1-k}(C_{n}\oplus C_{mn})=mn+n-1+k,\quad  \mbox{ for $k\in [0,n-1]$.}\end{align*}
The associated inverse question is to  characterize all extremal sequences of maximal length
$mn+n-2+k$ with $0\notin \Sigma_{\leq mn+n-1-k}(S)$. For $k=0$, this means  characterizing all zero-sum free sequences of maximal length $mn+n-2=\mathsf D(G)-1$.
For $k=1$, this means  characterizing all minimal zero-sum sequences of maximal length $mn+n-1=\mathsf D(G)$.  For $k=n-1$, this means  characterizing all extremal sequences of length  $mn+2n-3=\eta(G)-1$ with $0\notin \Sigma_{\leq mn}(S)$.
The precise structure in all three of these cases is known and summarized in Conjecture \ref{conj-shortzs} below.

For $k\leq 1$, this  was an involved undertaking achieved by combining the individual results of Gao, Geroldinger, Grynkiewicz, Reiher and  Schmid from \cite{propBGaoGer-multof2} \cite{PropB} \cite{Schmid-propB} \cite{PropBfix} \cite{Reiher-propB}  with the numerical verification of the case when $m=1$ and $n=9$ \cite{PropB-smallcase}. This characterization has since proved quite useful,  being employed in the proofs of several other results, e.g., \cite{propB-app4}  \cite{PropB-app6} \cite{PropB-app7}  \cite{PropB-app2} \cite{propB-app1} \cite{schmid-propC-exact} \cite{PropB-app8}.

For $k=n-1$, this was accomplished by Schmid \cite{schmid-propC-exact}. For simplicity, we assume $m=1$ in the following discussion. The  group $G=C_n\oplus C_{n}$ has \textbf{Property C} if every sequence $S$ with $|S|=\eta(G)-1=3n-3$ and $0\notin \Sigma_{\leq n}(S)$  must have the form $S=e_1^{[n-1]}\bdot e_2^{[n-1]}\bdot e_3^{[n-1]}$. It was shown in \cite{PropB->PropC} that, assuming  Conjecture \ref{conj-shortzs} holds for $k=1$ in $G$ (meaning, assuming the structure of minimal zero-sums of length $2n-1$ were known), then Property C holds for $G$. Once this case in Conjecture \ref{conj-shortzs} was resolved (as described above), this meant Property C was established without condition. However, it was a surprisingly nontrivial question to determine which $e_1,e_2,e_3\in G$ would give rise to a sequence $S=e_1^{[n-1]}\bdot e_2^{[n-1]}\bdot e_3^{[n-1]}$ with $0\notin \Sigma_{\leq n}(S)$.
 For $n=p$ prime, a derivation of the precise characterization from Property C can be found in \cite{emde-propC-prime}, and the derivation of the precise characterization from Property C in the general case (when $n$ may be composite) follows as a particular case of a more general result of  Schmid \cite{schmid-propC-exact}. The exact formulation is stated in Conjecture \ref{conj-shortzs}.4.
In Section \ref{sec-k=n-1}, we give a short alternative proof of this case,  deriving the precise characterization given in Conjecture \ref{conj-shortzs}.4  from Property C using the arguments from \cite{uzi}.

The structure of sequences $S$ of terms from $C_n\oplus C_n$ with $|S|=2n-2+k$ but $0\notin \Sigma_{\leq 2n-1-k}(S)$ was studied in \cite{S_<k-invariant2/3p} \cite{chaoI} for $k\in [2,n-2]$. In \cite{S_<k-invariant2/3p}, the case when $n$ is prime and $k\leq  \frac{2n+1}{3}$ was resolved, showing all such sequences must have the form $$S=e_1^{[n-1]}\bdot e_2^{[n-1]}\bdot (e_1+e_2)^{[k]}$$ for some basis $(e_1,e_2)$ for $G=C_n\oplus C_n$.
It was conjectured in \cite{S_<k-invariant2/3p} \cite{kevin-S_<k-invariant} that this should also hold for $k\in [2,n-2]$, and the results of \cite{chaoI} extended this conjecture to general $n$. There, a multiplicative property for the conjecture was established, showing  that, if the conjectured structure holds for $k_m$ in $C_m\oplus C_m$ and for $k_n$ in $C_n\oplus C_n$, where $k_m\in [0,m-1]$ and $k_n\in [0,n-1]$, then the conjectured structure also holds for $k=k_mn+k_n$ in $C_{mn}\oplus C_{mn}$. This reduced the characterization problem in $C_n\oplus C_n$ to the case when $n$ is prime.

The characterization in the case $C_{n}\oplus C_{mn}$ with $m\geq 2$, even including a precise statement of the potential structure for sequences of length $mn+n-2+k$ avoiding a nontrivial zero-sum of length at most $mn+n-1-k$, was completely open. Towards this goal, we continue with the following conjecture, summarizing known cases and extending previous conjectures for $m=1$ to $m\geq 2$.
As discussed above,  Parts 1, 2 and 4 in Conjecture \ref{conj-shortzs} are known, and Part 3 holds for $m=1$ when $n=p$ is prime and $k\leq \frac{2p+1}{3}$.   In general, we say that Conjecture \ref{conj-shortzs} holds for $k$ in $C_n\oplus C_{mn}$ if Conjecture \ref{conj-shortzs} is true when $G=C_n\oplus C_{mn}$ for the given value $k\in [0,n-1]$.
If true, Conjecture \ref{conj-shortzs}
 would fully characterize the structure of all extremal sequences for the zero-sum invariant $\mathsf s_{\leq mn+n-1-k}(G)$ over a general rank two abelian group $G=C_n\oplus C_{mn}$.

\begin{conjecture}
\label{conj-shortzs} Let $n\geq 2$ and $m\geq 1$,  let $G=C_n\oplus C_{mn}$,  let $k\in [0,n-1]$, and let $S$ be a sequence of terms from $G$ with  $$|S|=mn+n-2+k\quad\und\quad 0\notin \Sigma_{\leq mn+n-1-k}(S).$$ Then there exists either  a basis  $(e_1,e_2)$ for $G$ with $\ord(e_2)=mn$ or  a generating set $\{g_1,g_2\}$ for $G$ with $\ord(g_1+g_2)=mn$ such that the following hold.
\begin{itemize}
\item[1.]  If $k=0$, then $S\bdot g$ satisfies one of the possibilities given in Item 2,  where $g=-\sigma(S)$.
\item[2.]  If $k=1$, then either
\begin{itemize}
\item[(a)] $S=e_1^{[n-1]}\bdot \prod_{i\in [1,mn]}^\bullet (x_ie_1+e_2)$, for some $x_1,\ldots,x_{mn}\in [0,n-1]$ with $x_1+\ldots+x_{mn}\equiv 1\mod n$,
\item[(b)] $S=e_2^{[mn-1]}\bdot \prod_{i\in [1,n]}^\bullet (x_ie_2+e_1)$, for some $x_1,\ldots,x_{n}\in [0,mn-1]$ with $x_1+\ldots+x_{n}\equiv 1\mod mn$,
\item[(c)] $S=g_1^{[sn-1]}\bdot (g_1+g_2)^{[(m-s)n+1]}\bdot \prod_{i\in [1,n-1]}^\bullet (-x_ig_1+g_2)$, for some $s\in [1,m-1]$ and $x_1,\ldots,x_{n-1}\in [-1,n-2]$ with $x_1+\ldots+x_{n-1}=0$ and $ng_2=0$, or
\item[(d)]  $S=g_1^{[n-1]}\bdot (g_1+g_2)^{[(m-1)n+1]}\bdot\prod_{i\in [1,n-1]}^\bullet (-x_ig_1+g_2)$, for some  $x_1,\ldots,x_{n-1}\in [-1,n-2]$ with $x_1+\ldots+x_{n-1}=0$,
\end{itemize}

\item[3.] If $k\in [2,n-2]$, then either
\begin{itemize}
\item[(a)]$S=e_1^{[n-1]}\bdot e_2^{[sn-1]}\bdot(e_1+e_2)^{[(m-s)n+k]}$, for some $s\in [1,m]$, or
\item[(b)] $S=g_1^{[n-1]}\bdot g_2^{[n-1]}\bdot(g_1+g_2)^{[(m-1)n+k]}$.
\end{itemize}
\item[4.] If $k=n-1$, then either
\begin{itemize}
\item[(a)] $S=e_1^{[n-1]}\bdot e_2^{[sn-1]}\bdot(xe_1+e_2)^{[(m-s)n+n-1]}$, for some $s\in [1,m]$ and $x\in [1,n-1]$ with $\gcd(x,n)=1$, or
\item[(b)] $S=g_1^{[n-1]}\bdot g_2^{[n-1]}\bdot (g_1+g_2)^{[(m-1)n+n-1]}$.
\end{itemize}
\end{itemize}
Moreover, if $m=1$, then $(a)$ holds in all the above parts.
\end{conjecture}

The main goal of this paper is Theorem \ref{thm-mult-offdiag},  which shows that the structural description given in Conjecture  \ref{conj-shortzs}.3 is multiplicative in the sense that, if it holds for $k$ in $C_n\oplus C_n$, then it holds for $k$ in $C_n\oplus C_{mn}$. This reduces the characterization in a general rank two abelian group to the case $C_n\oplus C_n$, which in turn is reduced to the case $C_p\oplus C_p$ with $p\geq 11$ prime by the results of \cite{chaoI}.
The reduction to the case $C_n\oplus C_n$ is the main aim of the paper and emulates the strategy successfully used to characterize the extremal sequences for the Davenport Constant (the case $k\leq 1$), where the characterization problem was reduced to case $C_n\oplus C_n$ by Schmid \cite{Schmid-propB} and resolved in this case by the results from  \cite{propBGaoGer-multof2} \cite{PropB} \cite{PropBfix} \cite{Reiher-propB} (as well as the case $n=9$ \cite{PropB-smallcase}).

\begin{theorem}\label{thm-mult-offdiag}
Let $m,\,n\geq 2$. If Conjecture \ref{conj-shortzs} holds for $k\in [0,n-1]$ in $C_n\oplus C_{n}$, then Conjecture \ref{conj-shortzs} hold for $k$ in $C_n\oplus C_{mn}$.
\end{theorem}

The reduction to the diagonal case $C_n\oplus C_n$ is our main goal. However, combining Theorem \ref{thm-mult-offdiag} with known cases in Conjecture \ref{conj-shortzs} yields many group $C_{n}\oplus C_{mn}$ where the structure of extremal sequences is determined here without restriction. As several examples, we list the following corollaries.

\begin{corollary}\label{cor-mult-2,3}
 If $m\geq 1$ and $n=2^{s_1}3^{s_2}5^{s_3}7^{s_4}\geq 2$ with $s_1,s_2,s_3,s_4\geq 0$, then Conjecture \ref{conj-shortzs} holds  in $C_{n}\oplus C_{mn}$ for all  $k\in [0,n-1]$.
\end{corollary}

\begin{corollary}\label{cor-mult}
For any prime power $n\geq 2$ and $m\geq 1$, Conjecture \ref{conj-shortzs} holds  in $C_{n}\oplus C_{mn}$ for all  $k\leq \frac{2n+1}{3}$.
\end{corollary}

\begin{corollary}\label{cor-bigspec}
For $n\geq 4$ composite, $m\geq 1$ and $d\mid n$ a proper, nontrivial divisor,  Conjecture \ref{conj-shortzs} holds for $k=n-d-1$ and  for $k=n-2d+1$  in $C_{n}\oplus C_{mn}$.
\end{corollary}

\section{The case $k=n-1$}\label{sec-k=n-1}

As noted in the introduction, Conjecture \ref{conj-shortzs} holding for $k=1$ in $G=C_n\oplus C_n$ implies that Property C holds for $G$, meaning any sequence $S$ of $3n-3$ terms from $G=C_n\oplus C_n$ with $0\notin\Sigma_{\leq n}(S)$ must have the form  $S=e_1^{[n-1]}\bdot e_2^{[n-1]}\bdot e_3^{[n-1]}$ for some distinct $e_1,e_2,e_3\in G$.
The goal of this section is to give a new proof  of the  characterization of which elements $e_1,e_2,e_3\in G$ result in  a sequence $S=e_1^{[n-1]}\bdot e_2^{[n-1]}\bdot e_3^{[n-1]}$ with $0\notin \Sigma_{\leq n}(S)$.
Clearly, $0\notin \Sigma_{\leq n}(S)$ ensures $\ord(e_1)=\ord(e_2)=\ord(e_3)=n$. Thus there is some $f_1\in G$ such that $(f_1,e_2)$ is a basis for $G$. Letting $e_1=xf_1+ye_2$, we see that $(e_1,e_2)$ is a basis for $G$ unless $\gcd(x,n):=n/h>1$. However, if this were the case, then $T=e_1^{[h]}\bdot e_2^{[xh]}$ is a zero-sum subsequence of $S$ for some $x\in [0,\frac{n}{h}-1]$ having length $|T|= h+xh\leq n$, contradicting that $0\notin \Sigma_{\leq n}(S)$. Therefore  $(e_1,e_2)$ is a basis for $G$, and likewise  $(e_1,e_3)$ and $(e_2,e_3)$ must also be bases for $G$. However, obtaining further restriction on $e_1$, $e_2$ and $e_3$ is much less trivial. We begin with the following lemma showing how the characterization is related to a statement involving the index (see \cite{uzi}) of the sequence $(-x_1)\bdot (-x_2)\bdot 1$, where $e_3=x_1e_1+x_2e_2$, and continue afterwards with a series of lemmas modifying slightly the main line of argument for the prime case from \cite{uzi}

\begin{lemma}\label{lem-zs3case}
Let $G=C_n \oplus C_n$ with $n\geq 2$.  Suppose $(e_1,e_2)$ is a basis for $G$ and $S=e_1^{[n-1]}\bdot e_2^{[n-1]}\bdot (x_1e_1+x_2e_2)^{[n-1]}$, where $x_1,x_2\in [0,n-1]$. Then $0\notin \Sigma_{\leq n}(S)$ if and only if
$(-x_1 k)_n +(-x_2k)_n+(k)_n> n$ for every $k\in [1, n-1]$.
\end{lemma}

\begin{proof}
Consider an arbitrary zero-sum subsequence $T$ of $S$ and let $T=e_1^{[k_1]}\bdot  e_2^{[k_2]}\bdot  (x_1 e_1+x_2 e_2)^{[k]}$, where $k_1, k_2, k \in [0,n-1]$.
Note we cannot have $k=0$ (assuming $T$ nontrivial) as $(e_1,e_2)$ is a basis, so $k\in [1,n-1]$.
Then we have $k_1=(-x_1 k)_n$ and $k_2=(-x_2 k)_n$. Conversely, given any $k\in [1,n-1]$, the subsequence $T$ defined above with $k_1=(-x_1 k)_n$ and $k_2=(-x_2 k)_n$ will be a nontrivial zero-sum. Now  $$|T|=k_1+k_2+k= (-x_1 k)_n + (-x_2 k)_n + ( k)_n.$$
If $(-x_1 k)_n + (-x_2 k)_n + (k)_n\leq n$ for some $k\in [1,n-1]$, then the corresponding subsequence $T$ defined using $k$ is a nontrivial zero-sum subsequence of length at most $n$, showing $0\in \Sigma_{\leq n}(S)$. On the other hand, if $0\in \Sigma_{\leq n}(S)$, then there is a nontrivial zero-sum of the form  $T=e_1^{[k_1]}\bdot  e_2^{[k_2]}\bdot  (x_1 e_1+x_2 e_2)^{[k]}$, for some $k\in [1,n-1]$, which satisfies $|T|=(-x_1 k)_n + (-x_2 k)_n + ( k)_n \leq n$.
\end{proof}

The following is special case of \cite[Proposition 2.1.1]{uzi}.

\begin{lemma}\label{lem-length3}
Let $n\geq 2$ and let $x_1,x_2,x_3\in \Z$ with $x_1+x_2+x_3\equiv 0 \mod n$. There there exists $k\in [1,n-1]$ with $\gcd(k,n)=1$ and $(kx_1)_n+(kx_2)_n+(kx_3)_n\leq n$.
\end{lemma}

\begin{lemma}\label{lem-X}Let $n\geq 2$.
For $x\in [1, n-1]$ with $\gcd(x,n)=1$, let
\[
X(x)=\left\{  \left\lceil \frac{n}{x}\right\rceil, \left\lceil\frac{2n}{x}\right\rceil,\dots, \left\lceil\frac{(x-1)n}{x}\right \rceil \right\}\subseteq [2,n-1].
\]
\begin{enumerate}
	\item[1.] $|X(x)|=x-1$.
\item[2.] Let $d=\lfloor \frac{n}{x}\rfloor$. The difference between any two consecutive elements in $ X(x)$ is either $d$ or $d+1$ with $\min X(x)=d+1$ and $\max X(x)=n-d$ (for $x\geq 2$).
\item[3.] $[2,n-1]=X(x)\cup X(n-x)$ is a disjoint union.

	\item[4.] Let $\Delta (u,x)=(ux)_n - ((u-1)x)_n$. For every $u\in [1,n-1]$, $\Delta(u,x)\in \{x,x-n\}$ with $u\in X(x)$ iff $\Delta(u,x)=x-n$.
\end{enumerate}
\end{lemma}
\begin{proof}
Item 1--3 are given in \cite[Lemma 2.4]{uzi}. For Item 4, we have $-n<\Delta(u,x)<n$, and since $\Delta(u,x)\equiv x\mod n$, it follows  that $\Delta(u,x)\in \{x,x-n\}$.
Let $u\in \Z$ and  $t=\lfloor \frac{ux}{n}\rfloor$, so
$$tn\leq ux< (t+1)n\quad\und\quad ux=(ux)_n+tn=((u-1)x)_n+\Delta(u,x)+tn.$$ Hence
 $\Delta(u)=x-n$ when  $(u-1)x=((u-1)x)_n+(t-1)n< tn$, and $\Delta(u)=x$ when $(u-1)x=((u-1)x)_n+tn\geq tn$. Thus  $\Delta(u)=x-n$ if and only if $(u-1)x<tn$. Since $tn\leq ux$ always holds, this is equivalent to
  $\frac{tn}{x}\leq  u<\frac{tn}{x}+1$, and as $u$ is an integer, this is equivalent to $u=\lceil\frac{tn}{x}\rceil$. Therefore $\Delta(u,x)=x-n$ if and only if $u=\lceil\frac{tn}{x}\rceil$, where $t=\lfloor \frac{ux}{n}\rfloor$. Now restrict to $u\in [1,n-1]$.
  Since $u<n$, we have $t=\lfloor\frac{ux}{n}\rfloor\leq  x-1$.
  Since $u,\,x\geq 1$, we have $t=\lfloor\frac{ux}{n}\rfloor\geq 0$. However, $u=\lceil\frac{tn}{x}\rceil$ with $u\geq 1$ forces $t\neq 0$, while $u=\lceil\frac{tn}{x}\rceil=\frac{tn+r}{x}$ with $t\in [1,x-1]$ (and $r\in [0,x-1]$) ensures $t=\frac{ux-r}{n}=\lfloor\frac{ux}{n}\rfloor$ (as $0\leq r<x\leq n$). As a result, for $u\in [1,n-1]$, we find that $\Delta(u,x)=x-n$ if and only if $u=\lceil\frac{tn}{x}\rceil$ for some $t\in [1,x-1]$, as desired.
\end{proof}

\begin{lemma}\label{x-to-n+1} Let $n\geq 2$ and  let  $x_1,x_2,x_3\in \Z$ with  $\gcd(x_i,n)=1$ for all $i\in [1,3]$. If  $(kx_1)_n+(kx_2)_n+(kx_3)_n> n$ for every $k\in [1,n-1]$,   then  $x_i+x_j\equiv 0\mod n$ for some distinct $i,j\in [1,3]$.
\end{lemma}

\begin{proof}
If $\gcd(x_1+x_2+x_3,n)\neq 1$, then there exists some $r\in [1,n-1]$ such that $rx_1+rx_2+rx_3\equiv 0\mod n$. Applying Lemma \ref{lem-length3} to $rx_1,rx_2,rx_3\in \Z$, we find some $s\in [1,n-1]$ with $\gcd(s,n)=1$ and $(srx_1)_n+(srx_2)_n+(srx_3)_n\leq n$. But then, setting $k=(sr)_n\in [0,n-1]$, we have $(kx_1)_n+(kx_2)_n+(kx_3)_n\leq n$ with $k=(sr)_n\neq 0$ since $r\in [1,n-1]$ and $\gcd(s,n)=1$, which is contrary to hypothesis. Therefore, we conclude that $\gcd(x_1+x_2+x_3,n)=1$. As a result, replacing $x_1,x_2,x_3\in \Z$ with $sx_1,sx_2,sx_3\in \Z$, where $s\in \Z$ is an integer congruent to the inverse of $x_1+x_2+x_3$ modulo $n$, we can w.l.o.g. assume $$x_1+x_2+x_3\equiv 1\mod n.$$ Replacing each $x_i$ by $(x_i)_n$, we can w.l.o.g. assume $x_1,x_2,x_3\in [1,n-1]$. If $n=2$, then $x_1=x_2=x_3=1$, and the lemma holds. Therefore we may assume $n\geq 3$.

\subsection*{Claim 1:} For every $u\in [1,n-1]$,
\[
(ux_1)_n+(ux_2)_n+(ux_3)_n=n+u.
\]
\begin{proof}
Indeed, $(ux_1)_n+(ux_2)_n+(ux_3)_n\equiv u(x_1+x_2+x_3)\equiv u$, so $(ux_1)_n+(ux_2)_n+(ux_3)_n=kn+u$ for some $k\in\{0,1,2\}$. By hypothesis $kn+u\geq n+1$, so $(ux_1)_n+(ux_2)_n+(ux_3)_n\in \{n+u,2n+u\}$ for every $u\in [1,n-1]$.
Since $u\in [1,n-1]$ and $\gcd(x_i,n)=1$ for every $i\in [1,3]$, we have $(ux_i)_n\neq 0$ and $((n-u)x_i)_n=n-(ux_i)_n$ for all $i\in [1,3]$. Consequently,  if  $(ux_1)_n+(ux_2)_n+(ux_3)_n=2n+u$ for some $u\in [1,n-1]$, then we find $((n-u)x_1)_n+((n-u)x_2)_n+((n-u)x_3)_n=\big(n-(ux_1)_n\big)+\big(n-(ux_1)_n\big)+\big(n-(ux_1)_n\big)=3n-(2n+u)=n-u\leq n$, which is contrary to hypothesis. Claim 1 now follows.
\end{proof}

Let $X_i=X(x_i)$ and $\Delta(u,x_i)$ for $i\in [1,3]$ be as  defined in Lemma \ref{lem-X}. The special case $u=1$ in Claim 1 ensures $$x_1+x_2+x_3=n+1.$$ Hence, if $x_k=1$ for some $k\in [1,3]$, then the desired conclusion follows with $\{i,j\}=[1,3]\setminus k$, so we may assume \be\label{xiBig}x_1,\,x_2,\,x_3\geq 2,\ee ensuring $X_1$, $X_2$ and $X_3$ are each nonempty (by Lemma \ref{lem-X}.1).

\subsection*{ Claim 2.} $[2,n-1]=X_1\cup X_2\cup X_3$  is a disjoint union.

\begin{proof}
By Claim 1, for every $u\in [2,n-1]$, we have
\[
\Sum{i=1}{3} \Delta(u,x_i)=\Sum{i=1}{3}(ux_i)_n -\Sum{i=1}{3}((u-1)x_i)_n=(n+u)-(n+u-1)=1.
\]
On the other hand, if $u$ belongs to $s\in [0,3]$ of the sets $X_1$, $X_2$ and $X_3$, then Lemma \ref{lem-X}.4 implies  $1=\Sum{i=1}{3}\Delta(u,x_i)=x_1+x_2+x_3-sn=1+(1-s)n$, forcing $s=1$, which completes the claim as $X_i=X(x_i)\subseteq [2,n-1]$ holds trivially for any $x_i\in [1,n-1]$.
\end{proof}

Note that $2\in X_i=X(x_i)$ precisely when $\lceil \frac{n}{x_i}\rceil=2$, i.e., when $\frac{n}{2}\leq  x_i\leq n-1$. Consequently, in view of Claim 2 and $n\geq 3$, we can w.l.o.g. assume $$x_2,x_3<\frac{n}{2}\leq x_1.$$
Let  $x_1'=n-x_1\in [1,\frac{n}{2}]$ and set $X_1'=X(x_1')$. By Lemma \ref{lem-X}.3 and Claim 2, we have \be\label{pawclaw}X'_1=[2,n-1]\setminus X_1=X_2\cup X_3,\ee with the union disjoint. Hence, since $X_2$ and $X_3$ are nonempty, as noted above, it follows that $X'_1$ is also nonempty, ensuring $x'_1\geq 2$ (by Lemma \ref{lem-X}.1).

As in Lemma \ref{lem-X}, let $d_2=\lfloor \frac{n}{x_2}\rfloor$, $d_3=\lfloor \frac{n}{x_3}\rfloor$ and $d_1'=\lfloor \frac{n}{x_1'}\rfloor$. Note  $d_2, d_3,d'_1 \geq 2$ since $x_2,x_3,x'_1\leq \frac{n}{2}$. By Lemma \ref{lem-X}.2, the minimal element in $X_i$ for $i=2,3$ (resp. the minimal element in $X_1'$) is $d_i+1$ (resp. $d_1'+1$). Since $ X_1'=X_2\cup X_3$, the minimal element $d_1'+1$ in $X'_1$ must either equal the minimal element in $X_2$ or the minimal element in $X_3$, say w.l.o.g. the former, in which case $d'_1+1=d_2+1\in X_2$. Denote this joint value by $d:=d'_1=d_2$.

\subsection*{Claim 3:} $X_1'=X_2$.

\begin{proof}
  By \eqref{pawclaw}, we have  $X_2\subseteq X_1'$, and the first element of $X_1'$ is in $X_2$ by assumption, equal to the first element of $X_2$. Let $d+1=z_1<z_2<\dots<z_{x'_1-1}$ denote the elements of $X_1'$. Assume $z_k\in X_2$ for $k<x_1'-1$. By Lemma \ref{lem-X}.2 applied to $X_1'$, since $d\geq 2$, $z_{k+1}$ is one of the values $z_k+d$ or $z_k+d+1$, and exactly one of these elements is in $X_1'$. But $d_2=d$, so Lemma \ref{lem-X}.2 applied to $X_2$ yields that the next element of $X_2$ after $z_k\in X_2$ is also either $z_k+d$ or $z_k+d+1$. Since only one of these two possibilities lies in $X'_1$, namely the value between them equal to $z_{k+1}$, we are forced to conclude from $X_2\subseteq X'_1$ that the next element in $X_2$ after $z_k\in X_2$ is  $z_{k+1}\in X_2$. This shows, via  induction on $k$, that $X_2$ equals the first $|X_2|$ elements of $X'_1$. However, Lemma \ref{lem-X}.2 implies that $\max X_2=n-d_2=n-d=n-d'_1=\max X'_1$, which combined with the previous conclusion forces $X'_1=X_2$.
 \end{proof}

Since $X'_1=X_2\cup X_3$ is a disjoint union, we conclude from Claim 3 that $X_3=\emptyset$, whence  $x_3-1=|X_3|=0$ by Lemma \ref{lem-X}.1, contradicting \eqref{xiBig}.
\end{proof}

\begin{proposition}\label{thm-n-1-case}For any $n\geq 2$,
Conjecture \ref{conj-shortzs} holds for $k=n-1$ in $C_n\oplus C_n$.
\end{proposition}

\begin{proof}
Let $G=C_n\oplus C_n$ and let $S\in \Fc(G)$ be a sequence with $|S|=3n-3$ and $0\notin \Sigma_{\leq n}(S)$. Then $S=e_1^{[n-1]}\bdot e_2^{[n-1]}\bdot e_3^{[n-1]}$ for some $e_1,e_2,e_3\in G$ which pairwise form bases for $G$, as noted at the start of Section \ref{sec-k=n-1}. Write $e_3=x_1e_1+x_2e_2$ with $x_1,\,x_2\in [1,n-1]$. By Lemma \ref{lem-zs3case}, the hypothesis that $0\notin \Sigma_{\leq n}(S)$ is equivalent to $(-x_1\cdot k)_n+(-x_2\cdot k)_n+(k\cdot 1)_n>n$ holding for all $k\in [1,n-1]$. Since $(e_1,e_3)$ and $(e_2,e_3)$ are both bases, we must have $\gcd(x_1,n)=\gcd(x_2,n)=1$. Applying Lemma \ref{x-to-n+1} using the elements $-x_1$, $-x_2$ and $1$, we deduce that either $x_1+x_2\equiv 0\mod n$, or $x_1=1$, or $x_2=1$. If $x_1=1$, then the conclusion of Conjecture \ref{conj-shortzs} holds using the basis $(e_2,e_1)$. If $x_2=1$, then the conclusion of Conjecture \ref{conj-shortzs} holds using the basis $(e_1,e_2)$. If $x_1+x_2\equiv 0\mod n$, then $e_3=x_1e_1-x_1e_2$. In this final case, letting $y\in [1,n-1]$ be the multiplicative inverse of $x_1$ modulo $n$, we find that $ye_3=e_1-e_2$ and $e_1=ye_3+e_2$, in which case Conjecture \ref{conj-shortzs} holds using the basis $(e_3,e_2)$, which completes the proof.
\end{proof}

\section{Reduction to the Diagonal Case}

The goal of this section is to prove Theorem \ref{thm-mult-offdiag} as well as the three corollaries giving examples where our results confirm new cases in Conjecture \ref{conj-shortzs} without restriction. For the proof, we will need the following characterization of maximal length minimal zero-sums in a cyclic group, which is a direct consequence of \cite[Theorem 1]{hamconj} \cite[Theorem A]{chaoI} applied to $\Sigma_{n}(S\bdot 0^{[n-1]})=\Sigma(S)$ using $k=3$, where $S\in \Fc(C_n)$ is a zero-sum free sequence of maximal length $n-1$ (the case $n=2$ is trivial).

\begin{theirtheorem}
\label{thm-cyclic-char} Let $n\geq 2$, let $G=C_n$. If $S\in \Fc(G)$ is a minimal zero-sum sequence of length $|S|=\mathsf D(G)=n$, then $S=g^{[n]}$ for some $g\in G$ with $\ord(g)=n$. If $S\in \Fc(G)$ is a zero-sum free sequence of length $|S|=\mathsf D(G)-1=n-1$, then $S=g^{[n-1]}$ for some $g\in G$ with $\ord(g)=n$
\end{theirtheorem}

We continue with  the proof of Theorem \ref{thm-mult-offdiag}.

\begin{proof}[Proof of Theorem \ref{thm-mult-offdiag}]
Let $G=C_n\oplus C_{mn}$ and let $\varphi: G\rightarrow G$ be a homomorphism with$$\ker \varphi = C_m\quad\und\quad \varphi(G)= C_n\oplus C_n.$$ Since Conjecture \ref{conj-shortzs} is known to hold for $k\leq 1$ and $k=n-1$, we can assume $k\in [2,n-2]$ with $n\geq 4$.
Let $S\in \Fc(G)$ be a sequence with \be\label{hyp3}|S|=mn+n-2+k\quad\und\quad 0\notin \Sigma_{\leq mn+n-1-k}(S).\ee
Define a \emph{block decomposition} of $S$ to be a factorization $$S=W\bdot W_1\bdot\ldots \bdot W_{m-1}$$ with $1\leq |W_i|\leq n$  and $\varphi(W_i)$ zero-sum for each $i\in [1,m-1]$. Since $\mathsf s_{\leq n}(\varphi(G))=\mathsf s_{\leq n}(C_n\oplus C_n)=3n-2$ and $|S|=(m-2)n+3n-2+k\geq (m-2)n+3n-2$, it follows by repeated application of $\mathsf s_{\leq n}(\varphi(G))$ that $S$ has a block decomposition.

 \subsection*{Claim A} If $S =W\bdot W_1\bdot\ldots\bdot W_{m-1}$ is a block decomposition of $S$, then $|W_i|=n$ for all $i\in [1,m-1]$,  $|W|=2n-2+k$, $0\notin \Sigma_{\leq 2n-1-k}(\varphi(W))$,  and  $0\notin \Sigma_{\leq n-1}(\varphi(S))$. In particular, Conjecture \ref{conj-shortzs} holds for $\varphi(W)$.

\begin{proof}
Suppose $0\in \Sigma_{\leq 2n-1-k}(\varphi(W))$. Then there is a nontrivial subsequence $W_0\mid W$ with $|W_0|\leq 2n-1-k$ and $\varphi(W_0)$ zero-sum. Now $\sigma(W_0)\bdot \sigma(W_1)\bdot\ldots\bdot \sigma(W_{2m-2+k_m})$ is a sequence of $m$ terms from $\ker\varphi\cong C_m$. Since $\mathsf D(C_m)=m$, it follows that it has a nontrivial zero-sum subsequence, say $\prod_{i\in I}^\bullet \sigma(W_i)$ for some nonempty $I\subseteq [0,m-1]$. But then $\prod^\bullet_{i\in I}W_i$ is a nontrivial zero-sum subsequence of $S$ with $|{\prod}^\bullet_{i\in I}W_i|\leq (m-1)n+(2n-1-k)=mn+n-1-k$, contrary to \eqref{hyp3}. So we instead conclude that $0\notin \Sigma_{\leq 2n-1-k}(\varphi(W))$.

As a result, since $\mathsf s_{\leq 2n-1-k}(\varphi(G))=\mathsf S_{\leq 2n-1-k}(C_n\oplus C_n)=2n-1+k$, and since $|W_i|\leq n$ for all $i\in [1,m-1]$, it follows that $$2n-2+k=mn+n-2+k-(m-1)n\leq |S|-\Sum{i=1}{m-1}|W_i|=|W|\leq 2n-2+k,$$ forcing equality to hold in all estimates, i.e., $|W_i|=n$ for $i\in [1,m-1]$ and $|W|=2n-2+k$. If $0\in \Sigma_{\leq n-1}(\varphi(S))$, then we can find a nontrivial subsequence $W'_1\mid S$ with $|W'_1|\leq n-1$ and $\varphi (W'_1)$ zero-sum. Applying the argument used to show the existence of a block decomposition, we obtain a block decomposition $S=W'\bdot W'_1\bdot  \ldots\bdot W'_{m-1}$ with $|W'_1|\leq n-1$, contradicting what was just shown. Therefore $0\notin \Sigma_{\leq n-1}(\varphi(S))$. Finally, since $|W|=2n-2+k$ and $0\notin \Sigma_{\leq 2n-1-k}(\varphi(W))$ with Conjecture \ref{conj-shortzs} holding for $k$ in $C_n\oplus C_n$ by hypothesis, it follows that Conjecture \ref{conj-shortzs} holds for $\varphi(W)$, completing the claim.
\end{proof}

Suppose $$S=\wtilde W\bdot W_0\bdot W_1\bdot\ldots\bdot W_{m-1}$$ with each $\varphi(W_i)$ a nontrivial zero-sum for $i\in [0,m-1]$ and $|\wtilde W|\geq 2k-n-1$.
We call this a \emph{weak block decomposition} of $S$ with associated sequence $$S_\sigma=\sigma(W_0)\bdot \sigma(W_1)\bdot\ldots\bdot \sigma(W_{m-1})\in \Fc(\ker \varphi).$$
 Since $\ker \varphi=C_m$ and $|S_\sigma|=m=\mathsf D(C_m)$, it follows that $S_\sigma$ contains  a nontrivial zero-sum. In view of Claim A, we have $|W_i|\geq n$ for all $i\in [0,m-1]$. As a result, if $S_\sigma$ contained a proper, nontrivial zero-sum subsequence, then $S$ would have a nontrivial zero-sum of length at most $|S|-|\wtilde W|-n\leq (mn+n-2+k)-(2k-n-1)-n=mn+n-1-k$, contrary to \eqref{hyp3}. We conclude that the associated sequence $S_\sigma$ must be a minimal zero-sum of length $\mathsf D(C_m)=m$, in which case Theorem \ref{thm-cyclic-char} implies that there is some $g_0\in \ker \varphi$ with $\ord(g_0)=m$ such that
\begin{equation}\label{invcn}
\sigma(W_0) = \sigma(W_1) = \ldots = \sigma(W_{m-1}) = g_0.
\end{equation}

Now let $S=W\bdot W_1\bdot\ldots\bdot W_{m-1}$ be a fixed but otherwise arbitrary  block decomposition. In view of Claim A, we have $|W|=2n-2+k$ with $0\notin \Sigma_{\leq 2n-1-k}(\varphi(W))$ and Conjecture \ref{conj-shortzs} holding for $\varphi(W)$ using $k\in [2,n-2]$. As a result, there is a basis $(\overline e_1,\overline e_2)$ for $\varphi(G)\cong C_n\oplus C_n$ such that
 $$\varphi(W) = \overline e_1^{[n-1]} \bdot \overline e_2^{[n-1]} \bdot  (\overline e_1+\overline e_2)^{[k]},$$ and there is a subsequence $W_0\mid W$ with $$\varphi(W_0)=\overline e_1^{[n-k]}\bdot \overline e_2^{[n-k]}\bdot (\overline e_1+\overline e_2)^{[k]}.$$
Setting $\wtilde W=W\bdot W_0^{[-1]}$, so $$\varphi(\wtilde W)=\overline e_1^{[k-1]}\bdot \overline e_2^{[k-1]},$$ we find $|\wtilde W|=2k-2>2k-n-1$ (as $n\geq 2$), meaning $S=\wtilde W\bdot W_0\bdot W_1\bdot\ldots\bdot W_{m-1}$ is a weak block decomposition  with associated sequence $S_\sigma=\sigma(W_0)\bdot \sigma(W_1)\bdot\ldots\bdot \sigma(W_{m-1})$ satisfying \eqref{invcn} for some $g_0\in \ker\varphi$ with $\ord(g_0)=m$.

  \subsection*{Claim B} For any $j\in [1,m-1]$, if $W'_j\mid W\bdot W_j$ is a subsequence with $|W'_j|=n$ and $\varphi(W'_j)$ zero-sum, then $\varphi(W\bdot W_j\bdot (W'_j)^{[-1]})=\varphi(W)=\overline e_1^{[n-1]}\bdot \overline e_2^{[n-1]}\bdot (\overline e_1+\overline e_2)^{[k]}$.

  \begin{proof}
We can w.l.o.g. assume $j=1$. Setting $W'=W\bdot W_1\bdot (W'_1)^{[-1]}$, we find that $S=W'\bdot W'_1\bdot W_2\bdot \ldots\bdot W_{m-1}$ is a block decomposition, so Conjecture \ref{conj-shortzs} must hold for $\varphi(W')$ by Claim A with respect to some basis $(\overline f_1,\overline f_2)$, meaning $\varphi(W\bdot W_1\bdot (W'_1)^{[-1]})=\varphi(W')=\overline f_1^{[n-1]}\bdot \overline f_2^{[n-1]}\bdot (\overline f_1+\overline f_2)^{[k]}$. We need to show $\{\overline f_1,\overline f_2\}=\{\overline e_1,\overline e_2\}$.
Assuming by contradiction that this fails, we can w.l.o.g. assume $\overline f_1\notin \{\overline e_1,\overline e_2\}$, ensuring $f_1$ has multiplicity at least $n-1$ in $\varphi(W\bdot W_1)$.
Note, since the $n$-term zero-sum $\varphi(W_1)$ cannot contain a term with multiplicity exactly $n-1$, any term $g\notin \{\overline e_1,\overline e_2\}$ with multiplicity at least $n-1$ in $\varphi(W\bdot W_1)$ must either have  $\varphi(W_1)=g^{[n]}$ or else  $g=\overline e_1+\overline e_2$ with
$\vp_{\overline e_1+
  \overline e_2}(\varphi(W_1))\geq n-1-k\geq 1$.
In both cases, there cannot be a second term $g'\notin \{\overline e_1,\overline e_2\}$ with multiplicity at least $n-1$, the former since $\vp_{\overline e_1+\overline e_2}(W_0\bdot W_1)=k\leq n-2$, and the latter since $\varphi(W_1)$ cannot contain a term with multiplicity exactly $n-1$. As a result, we conclude that $\overline e_1$, $\overline e_2$ and $\overline f_1$ are the only terms   with multiplicity at least $n-1$ in $\varphi(W\bdot W_1)$, and thus w.l.o.g. $\overline f_2=\overline e_2$.

 Suppose $\overline f_1=\overline e_1+\overline e_2$. Then $\varphi(W')=(\overline e_1+
\overline e_2)^{[n-1]}\bdot \overline e_2^{[n-1]}\bdot (\overline e_1+2\overline e_2)^{[k]}$, and since $n\geq 4$ ensures $\overline e_1+2\overline e_2\neq \overline e_1$ with $\varphi(W'_1)$ an $n$-term zero-sum, it follows that  $\varphi(W'_1)=\overline e_1^{[n]}$. In such case, $\varphi(W_1)=(\overline e_1+\overline e_2)^{[n-1-k]}\bdot \overline e_1\bdot (\overline e_1+2\overline e_2)^{[k]}$, which is only zero-sum for $k=1$, contradicting that $k\in [2,n-2]$. So we conclude that $\overline f_1\neq \overline e_1+\overline e_2$, in which case we must instead have $\varphi(W_1)=\overline f_1^{[n]}$.

In this case, $\varphi(W')=\overline f_1^{[n-1]}\bdot\overline  e_2^{[n-1]}\bdot (\overline f_1+\overline e_2)^{[k]}$ with $\overline f_1+\overline e_2\in \{\overline e_1,\overline e_1+\overline e_2\}$. Thus either $\overline f_1=\overline e_1-\overline e_2$ or $\overline e_1$. Since $\overline f_1\neq \overline e_1$, this means $\overline f_1=\overline e_1-\overline e_2\neq \overline e_1+\overline e_2$ (as $n\geq 3$), $\varphi(W')=\overline f_1^{[n-1]}\bdot\overline  e_2^{[n-1]}\bdot
\overline e_1^{[k]}=(\overline e_1-\overline e_2)^{[n-1]}\bdot\overline  e_2^{[n-1]}\bdot
\overline e_1^{[k]}$ and $\varphi(W'_1)=\overline e_1^{[n-1-k]}\bdot (\overline e_1+\overline e_2)^{[k]}\bdot \overline f_1$. Hence,  since $\varphi(W'_1)$ is an $n$-term zero-sum, it follows that $\overline e_1-\overline e_2=\overline f_1= \overline e_1-k\overline e_2$, contradicting that $k\in [2,n-2]$, which completes the claim.
  \end{proof}

\subsection*{Claim C}
$\varphi(W_j)  \in \{\overline e_1^{[n]}, \overline e_2^{[n]}, (\overline e_1 + \overline e_2)^{[n]}\}$ for every $j\in [1,m-1]$.

\begin{proof}
Since each $\varphi(W_j)$, for $j\in [1,m-1]$, is a zero-sum of length $n$ by Claim A, it suffices to show $\supp(\varphi(W_1\bdot\ldots\bdot W_{m-1}))\subseteq \{\overline e_1,\overline e_2,\overline e_1+\overline e_2\}$. Let $g\in \supp(W_j)$ and $j\in [1,m-1]$ be arbitrary.  Since $0\notin \Sigma_{\leq n-1}(\varphi(S))$ by Claim A and $|W\bdot W_j\bdot g^{[-1]}|=3n-3+k\geq 3n-2=\mathsf s_{\leq n}(C_n\oplus C_n)$, there is subsequence $W'_j\mid W\bdot W_j\bdot g^{[-1]}$ with $|W'_j|=n$, $\varphi(W'_1)$ zero-sum  and  $g\in \supp(W\bdot W_j\bdot (W'_j)^{[-1]})$. Thus Claim B ensures that $\varphi(g)\in\{\overline e_1,\overline e_2,\overline e_1+\overline e_2\}$, and as $g\in \supp(W_j)$ and $j\in [1,m-1]$ were arbitrary, the claim follows.
\end{proof}

\subsection*{Claim D} There are $g_1,\,g_2\in G$ with $\varphi(g_1)=\overline e_1$, $\varphi(g_2)=\overline e_2$ and   $\supp(S)=\{g_1,g_2,g_1+g_2\}$.

\begin{proof}
Let $x\in \supp(\wtilde W)$ and $y\in \supp(W_j)$, for some $j\in [0,m-1]$, be arbitrary terms with $\varphi(x)=\varphi(y)=\overline e_1$, which exist as $k\geq 2$ and $\vp_{\overline e_1}(\varphi(W_0))=n-k\geq 1$. If we set $\wtilde W'=\wtilde W\bdot x^{[-1]}\bdot y$ and $W'_j=W_j\bdot y^{[-1]}\bdot x$, we find that $S=\wtilde W'\bdot W_1\bdot\ldots\bdot W_{j-1}\bdot W'_j\bdot W_{j+1}\bdot \ldots\bdot W_{m-1}$ is a weak product decomposition with associated sequence $S_\sigma \bdot \sigma(W_j)^{[-1]}\bdot \sigma(W'_j)=g_0^{[m-1]}\bdot \sigma(W'_j)$.
As a result, Theorem \ref{thm-cyclic-char} implies that $g_0-y+x=\sigma(W_j)-y+x=\sigma(W'_j)=g_0$, whence $x=y$. This shows that all terms $x\in \supp(S)$ with $\varphi(x)=\overline e_1$ are equal to the same element (say) $g_1\in G$.
The same argument shows that all terms $x\in \supp(S)$ with $\varphi(x)=\overline e_2$ are equal to the same element (say) $g_2\in G$. Let $z\in \supp(W_0\bdot \ldots \bdot W_{m-1})$ be an arbitrary term with $\varphi(z)=\overline e_1+\overline e_2$, which exists as $\vp_{\overline e_1+\overline e_2}(\varphi(W_0))=k\geq 2$. Let $g\in \supp(W_j)$ with  $j\in [0,m-1]$. Since $k\geq 2$, we have $g_1\bdot g_2\mid \wtilde W$. If we set $\wtilde W'=\wtilde W\bdot g_1^{[-1]}\bdot g_2^{[-1]}\bdot z$ and $W'_j=W_j\bdot z^{[-1]}\bdot g_1\bdot g_2$, we find that $S=\wtilde W'\bdot W_1\bdot\ldots\bdot W_{j-1}\bdot W'_j\bdot W_{j+1}\bdot \ldots\bdot W_{m-1}$ is a weak product decomposition with associated sequence $S_\sigma \bdot \sigma(W_j)^{[-1]}\bdot \sigma(W'_j)=g_0^{[m-1]}\bdot \sigma(W'_j)$. As a result, Theorem \ref{thm-cyclic-char} implies that $g_0-z+g_1+g_2=\sigma(W_j)-z+g_1+g_2=\sigma(W'_j)=g_0$, whence $z=g_1+g_2$. This shows that all terms $x\in \supp(W_0\bdot \ldots\bdot W_{m-1})$ with $\varphi(z)=\overline e_2$ are equal to the same element $g_1+g_2\in G$. Since, by definition of $\wtilde W$, there are no terms $z\in \supp(\wtilde W)$ with $\varphi(z)=\overline e_1+\overline e_2$, the claim follows.
\end{proof}

In view of Claim D and \eqref{invcn}, we find that \be\label{g-help}g_0=\sigma(W_0)=(n-k)g_1+(n-k)g_2+k(g_1+g_2)
=ng_1+ng_2.\ee
Since $\varphi(g_1+g_2)=\overline e_1+\overline e_2$ has order $n$, it follows that $n$ divides  $\ord(g_1+g_2)$, which, combined with $n(g_1+g_2)=g_0$ having $\ord(g_0)=m$, forces $\ord(g_1+g_2)=mn$.
If $\la g_1,g_2\ra=G'$ were a proper subgroup of $G=C_n\oplus C_{mn}$, then  $\ord(g_1+g_2)=mn$ ensures that $\la g_1,g_2\ra=C_{n'}\oplus C_{mn}$ for some proper divisor $n'\mid  n$. In such case, $S$ would be a sequence of terms from $G'=C_{n'}\oplus C_{mn}$ with length $|S|=(m+1)n-2+k\geq mn+2n'-2+k\geq \eta(G')$, ensuring that $S$ has a zero-sum of length at most $mn$, contrary to \eqref{hyp3}. Therefore $\la g_1,g_2\ra=G$, meaning $\{g_1,g_2\}$ is a generating set for $G$.

If there are $i,\,j\in [1,m-1]$ with $\varphi(W_i)=\overline e_1^{[n]}$ and $\varphi(W_j)=\overline e_2^{[n]}$, then Claim D ensures that $W_i=g_1^{[n]}$ and $W_j=g_2^{[n]}$. Combined with \eqref{invcn} and \eqref{g-help}, this  implies $ng_1+ng_2=g_0=\sigma(W_i)=ng_1$ and $ng_1+ng_2=g_0=\sigma(W_j)=ng_2$, whence $ng_1=ng_2=0$ and $g_0=ng_1+ng_2=0$, contradicting that $\ord(g_0)=m\geq 2$. Therefore we can w.l.o.g. assume $$W_1, W_2,\ldots, W_{m-1}\in \{ g_2^{[n]}, (g_1+g_2)^{[n]} \}.$$
If $W_1=\ldots=W_{m-1}=(g_1+g_2)^{[n]}$, then Conjecture \ref{conj-shortzs}.3(b) holds. So we can w.l.o.g assume $W_1=g_2^{[n]}$, in which case \eqref{invcn} and \eqref{g-help} yield $ng_1+ng_2=g_0=\sigma(W_1)=ng_2$, implying $$ng_1=0.$$

Let $s-1\in [1,m-1]$ be the number of $i\in [1,m-1]$ with $W_i=g_2^{[n]}$. Then
$$S=g_1^{[n-1]}\bdot  g_2^{[sn-1]}\bdot  (g_1+g_2)^{[(m-s)n+k]}.$$ Since $ng_1=0$, we have $\ord(g_1)\leq n$. Thus  $$n^2m=|G|=|\la g_1,g_2\ra|=\frac{|\la g_1\ra|\cdot |\la g_2\ra|}{|\la g_1\ra\cap \la g_2\ra|}\leq n(mn),$$ forcing  equality to  hold in all above estimates. As a result, $\ord(g_1)=n$, $\ord(g_2)=mn$ and $\la g_1\ra\cap \la g_2\ra=\{0\}$, meaning  $(g_1,g_2)$ is a basis for $G$, and now Conjecture \ref{conj-shortzs}.3(a) holds, completing the proof.
\end{proof}

To conclude the paper, we note that Corollaries \ref{cor-mult-2,3}, \ref{cor-mult} and \ref{cor-bigspec} following immediately by  combining Theorem \ref{thm-mult-offdiag} with the respective result from \cite[Corollaries 1.3, 1.4 and 1.5]{chaoI}.

\end{document}